\documentclass[a4paper,11pt]{article}

\usepackage{amssymb, amsmath, color}

\def\A{\mathbb{A}}
\def\C{\mathbb{C}}
\def\R{\mathbb{R}}
\def\N{\mathbb{N}}
\def\P{\mathbb{P}}
\def\Q{\mathbb{Q}}
\def\Z{\mathbb{Z}}

\def\Res {{\rm Res}}

\def\dist {{\rm dist}}
\def\bm{\boldsymbol}

\def\abs#1{\mathopen| #1 \mathclose|}	% use instead of $|x|$
\newtheorem{defn}{Definition}%[section]

\newtheorem{lemma}[defn]{Lemma}
\newtheorem{proposition}[defn]{Proposition}
\newtheorem{theorem}[defn]{Theorem}

\newtheorem{remark}[defn]{Remark}
\newtheorem{notation}[defn]{Notation}
\newtheorem{example}[defn]{Example}

           {\vspace{3.3mm}
           \noindent{\bf #1}\it}%
           {\vspace{3.3mm}}

\newenvironment{proof}[1]{
  \trivlist \item[\hskip \labelsep{\it #1}]}{\hfill\mbox{$\square$}
  \endtrivlist}

\newcommand{\Hom}{h}

\title{On the minimum of a polynomial function on a basic closed semialgebraic set and applications}
\author{Gabriela Jeronimo$^{\diamondsuit,}$\footnote{Partially supported by the following grants: PIP 099/11 CONICET and UBACYT 20020090100069 (2010/2012).}, Daniel Perrucci$^{{\diamondsuit},*}$, Elias Tsigaridas$^{\sharp,}$%
\footnote{Partially supported by an individual postdoctoral grant from the
Danish Agency for Science, Technology and Innovation,
and also acknowledges support from the Danish National Research Foundation and
the National Science Foundation of China (under the grant 61061130540)
for the Sino-Danish Center for the Theory of Interactive Computation,
within which part of this work was performed.}
\\[3mm]
{\small ${\diamondsuit}$ Departamento de Matem\'atica, Facultad de Ciencias Exactas y Naturales, Universidad de Buenos Aires}\\ {\small and IMAS, CONICET}\\
{\small $\sharp$ Computer Science Department, Aarhus University}
}

\begin{document}

\maketitle

\begin{abstract}
We give an explicit upper bound for the algebraic degree and an explicit lower bound for the absolute value of the minimum of a polynomial function on a compact connected component of a basic closed semialgebraic set when this minimum is not zero. As an application, we obtain a lower bound for the separation of two disjoint connected components of basic closed semialgebraic sets, when at least one of them is compact.
\end{abstract}

\section{Introduction}

Let $T\subset \R^n$ be a basic closed semialgebraic set defined by polynomials with integer coefficients and let $C$ be a compact connected component of $T$.
The first aim of this work is to find bounds $\delta>0$ and $b >0$ such that if the minimum value that a polynomial $g\in \Z[x_1,\dots, x_n]$ takes over $C$ is not zero, then it is an algebraic number of degree at most $\delta$ and its absolute value is greater or equal to $b$. We look for explicit bounds $\delta $ and $b$ in terms of the number of variables, the number of polynomials defining $T$ and  given upper bounds for the degrees and coefficient size of these polynomials and $g$.
Such explicit bounds are of fundamental importance in the complexity analysis of symbolic and numerical methods for optimization and polynomial system solving (see, for instance, \cite{Bajaj}).

A standard technique to handle optimization problems with inequality constraints  is to use the Karush-Kuhn-Tucker conditions (see \cite[Chapter 12]{NW}). In \cite{NR}, this approach is combined with deformation techniques  to obtain the algebraic degree of the minimizers in polynomial optimization over subsets of $\R^n$ defined by at most  $n$  polynomials under certain non-degeneracy assumptions.

In this paper, we consider the minimization problem for an arbitrary family of polynomial constraints. Since the system which gives the critical points for $g$ on $T$ may not satisfy the required hypothesis or may provide us with an infinite set of possible minimizers, we  use the deformation techniques in \cite{JPS10} which follow the spirit of \cite[Chapter 13]{BPR}. The deformation enables us to deal with `nice' systems which, in the limit, define a finite set of minimizing points. A careful analysis of the perturbed systems combined with resultant-based estimations relying on \cite{Sombra04} leads us to the explicit bounds (see  \cite {BR10}, \cite{EMT10}, \cite{HKL+} and \cite{JP10} for similar applications of these techniques).
Our main result is the following:

\begin{theorem}\label{minimum}
Let $T = \{x \in \R^n \ | \ f_1(x) = \dots = f_{l}(x) = 0, f_{l + 1}(x) \ge  0, \dots, f_{m}(x) \ge  0\}$ be defined by polynomials $f_1,\dots, f_m \in \Z[x_1,\dots, x_n]$ with degrees bounded by an even integer $d$ and coefficients of absolute value at most $H$, and let $C$ be a compact connected component of $T$. Let $g\in \Z[x_1,\dots, x_n]$ be a polynomial of degree $d_0\le d$ and coefficients of absolute value bounded by $H_0\le H$.
Then, the minimum value that $g$ takes over $C$ is a real algebraic number of degree at most
$$2^{n-1} d^n $$
and,
if it is not zero, its absolute value is greater or equal to
\begin{equation}\label{bigbound}
(2^{4-\frac{n}2}\tilde H d^n)^{-n2^nd^n},
\end{equation}
where $\tilde H =  \max\{H, 2n+2m\}$.
\end{theorem}

We also show that the previous result holds for non-compact connected components of $T$ having a compact set of minimizers for $g$ (see Theorem \ref{minimum_noncompact}).

Usually solutions of optimization problems are algebraic numbers,
hence it is natural to study the degree of the minimal polynomial that
defines them \cite{NRS-10}. Our bound for the degree can be seen as
an extension of the result in \cite{NR}. In addition, we present
an explicit lower bound for the absolute value of the
minimum. This bound can be applied, for instance, to get an explicit
upper bound for the degrees in Schm\"udgen's Positivstellensatz (see
\cite[Theorem 3]{Sch}).

A further application of our main result, which is in fact the original motivation of this work, is  an explicit lower bound for the separation between disjoint connected components of basic closed semialgebraic sets. Bounds of this kind can be applied to estimate the running time of numeric algorithms dealing with polynomial equations and inequalities (see, for instance, \cite{MMT11}, \cite{Yak05}).
For isolated points, the problem has already been studied both in the complex and real settings (see, for instance, \cite{Canny}, \cite{EMT10}, \cite{HKL+}).
 Our result, which includes positive dimensional situations, is the following:

\begin{theorem}\label{distance} Let
$T_1 = \{x \in \R^n \ | \ f_1(x) = \dots = f_{l_1}(x) = 0, f_{l_1 + 1}(x) \ge  0, \dots, f_{m_1}(x) \ge  0\}$,  $T_2 = \{x \in \R^n \ | \ g_1(x) = \dots = g_{l_2}(x) = 0, g_{l_2 + 1}(x) \ge  0, \dots, g_{m_2}(x) \ge  0\}$ be defined by polynomials $f_1, \dots, f_{m_1}, g_{1}, \dots, g_{m_2} \in \Z[X_1, \dots, X_n]$
with degrees bounded by an even integer $d$ and coefficients of absolute value at most $H$.
Let $C_1$ be a compact connected component of $T_1$ and $C_2$ a connected component of $T_2$. Then, if $C_1 \cap C_2 = \emptyset$, the distance between $C_1$ and $C_2$ is at least $$ (2^{4-n} \tilde H d^{2n})^{-n2^{2n}d^{2n}}$$ where $\tilde H = \max \{H, 4n+2m_1+2m_2\}$.
\end{theorem}

The paper is organized as follows: Section \ref{sec:minimum} is devoted to proving the bounds for the minimum. First, we introduce the deformation techniques we use and prove some geometric properties of this deformation which, in particular, enables us to give a characterization of minimizers as solutions to a polynomial system; then, we prove Theorem \ref{minimum}. In Section \ref{sec:distance}, we prove Theorem \ref{distance} and present an easy example to show that the double exponential nature of our bounds is unavoidable.

\section{The minimum of a polynomial function}
\label{sec:minimum}

Let $f_1, \dots, f_{m}, g \in \Z[x_1, \dots, x_n]$ with $n\ge 2$, $d$ an even positive integer such that $\deg (f_1), \dots, \deg (f_{m}) \le d$, and $d_0=\deg(g)\le d$. Let $H\in \N$ be an upper bound on the absolute values of all the coefficients of $f_1, \dots, f_{m}$ and $H_0\in \N$, $H_0\le H$, an upper bound on the absolute values of the coefficients of $g$.
Let $T = \{x \in \R^n \ | \ f_1(x) = \dots = f_{l}(x) = 0, f_{l + 1}(x) \ge  0, \dots, f_{m}(x) \ge  0\}$ and let $C$ be a compact connected component of $T$.

\subsection{The deformation}

Here we introduce some notation that we will use throughout this section. Let
\begin{itemize}

\item $A \in \Z^{(m+1)\times(n+1)}, A = (a_{ij})_{0 \le i \le m, \ 0 \le j \le n}$ be a matrix such that each of its submatrices has maximal rank and $a_{ij} > 0$ for every $i,j$.

\item For every $1 \le i \le m$, $\tilde f_i(x) = \sum_{j = 1}^na_{ij}x_j^{d} +a_{i0}$, $F_i^+(t,x) = f_i(x) + t\tilde f_i(x)$ and  $F_i^-(t,x) = f_i(x) - t\tilde f_i(x)$.

\item $\tilde g(x) = \sum_{j=1}^n a_{0j} x_j^{d} + a_{00}$ and $G(t,x) = g(x) + t\tilde g(x)$.

\item For every $S \subset \{1, \dots, m\}$ and
$\sigma \in \{+,-\}^{S}$,
$$
\hat W_{S, \sigma} = \{(t,x) \in \A \times \A^n \ | \ F_i^{\sigma_i}(t,x) = 0 \hbox{ for every } i \in S \},
$$
$$
\hat Z_{S, \sigma} = \{(t,x) \in \A \times\A^n \ | \ (t,x) \in \hat W_{S, \sigma} \hbox{ and }
\{ \nabla_{x}F_i^{\sigma_i}(t,x), i \in S \} \hbox{ is linearly dependent}
\},
$$
and
$$
\hat V_{S, \sigma} = \{(t,x,\lambda) \in \A \times\A^n \times \P^{\#S} \ | \ (t,x) \in \hat W_{S, \sigma} \hbox{ and }
$$
$$
\lambda_0\nabla_{x} G(t,x) = \sum_{i \in S} \lambda_i \nabla_{x}F_i^{\sigma_i}(t,x)
\},
$$
where $\A$ and $\P$ denote the affine and projective spaces over the complex numbers respectively.
There are
$\sum_{i=1}^{m}{m \choose i}2^i = 3^m - 1$ different sets $\hat W_{S, \sigma}$.
We consider the decomposition of  $\hat W_{S, \sigma}$ as $\hat W_{S, \sigma} = W_{S, \sigma}^{(0)} \cup W_{S, \sigma}^{(1)} \cup W_{S, \sigma}$, where
\begin{itemize}
\item $W_{S, \sigma}^{(0)}$ is the union of the irreducible components of $\hat W_{S, \sigma}$ included in ${t = 0}$,
\item $W_{S, \sigma}^{(1)}$ is the union of the irreducible components of $\hat W_{S, \sigma}$ included in ${t = t_0}$ for some $t_0 \in \C-\{0\}$,
\item $W_{S, \sigma}$ is the union of the remaining irreducible components of $\hat W_{S, \sigma}$,
\end{itemize}
and the analogous decompositions of
$\hat Z_{S, \sigma}$ and $\hat V_{S, \sigma}$ as $\hat Z_{S, \sigma} = Z_{S, \sigma}^{(0)} \cup Z_{S, \sigma}^{(1)} \cup Z_{S, \sigma}$ and $\hat V_{S, \sigma} = V_{S, \sigma}^{(0)} \cup V_{S, \sigma}^{(1)} \cup V_{S, \sigma}$
respectively.

\item For a group of variables  $y$, $\Pi_y$ will indicate the projection to the coordinates $y$.
\end{itemize}

\bigskip

We start by constructing a matrix $A$ satisfying the conditions required above and bounding their entries.

\begin{lemma}\label{smallmatrix} There exists a matrix
$A \in \Z^{(m+1)\times(n+1)}, A = (a_{ij})_{0 \le i \le m, \ 0 \le j \le n}$, such that each of its submatrices has maximal rank and $0 < a_{ij} \le 2(n  +m)$ for every $i,j$.
\end{lemma}
\begin{proof}{Proof.}
Let $p$ be a prime number such that $n+m+2 \le p \le 2n + 2m + 1$, which  exists by Bertrand's postulate.

Consider the Hilbert matrix $A_1 = (\frac{1}{i+j+1})_{0 \le i \le m, \ 0 \le j \le n}$, which is a particular case of a Cauchy matrix;  therefore, every submatrix of $A_1$ has maximal rank.
Let $A_2 = (n+m+1)!A_1$; then, $A_2 \in \Z^{(m+1)\times (n+1)}$ and
the positive prime factors of
every entry of $A_2$ are prime numbers lower than or equal to $n+m+1$. Looking at the formula for the determinant of Cauchy matrices, one can see that the determinant of every square submatrix of $A_2$ is an integer (different from $0$) such that all its prime factors are lower than or equal to $n+m+1$.

Finally, take $A$ as the matrix obtained by replacing every entry of $A_2$ by its remainder in the division by $p$, which is never equal to $0$. Then it is clear that $A$ has the required properties.
\end{proof}

Before proceeding, we will state two basic facts about the varieties previously defined. We postpone the proof of these results to Section \ref{obtthebound}.

\begin{lemma}\label{intersections} Let $S \subset \{1, \dots, m\}$ and
$\sigma \in \{+,-\}^{S}$. If $\#S > n$, the variety $W_{S,\sigma}$ is empty.
\end{lemma}

\begin{lemma} \label{lemma:li} For every $S \subset \{1, \dots, m\}$ and
$\sigma \in \{+,-\}^{S}$, the variety $Z_{S,\sigma}$ is empty.
\end{lemma}

\subsection{Geometric properties}\label{geomprop}

For every $t \ge 0$, let
\begin{displaymath}
  \begin{aligned}
    T_t = \{x \in \R^n \ |  &\ F^+_1(t,x) \ge 0,  \dots,  F^+_{l}(t,x) \ge 0, & F^+_{l + 1}(t,x) \ge  0, \dots, F^+_{m}(t,x) \ge  0, &\\
    & F^-_1(t,x) \le 0 , \dots,  F^-_{l}(t,x) \le 0  \}.
  \end{aligned}
\end{displaymath}

As $\tilde f_i (x) > 0 $ for every $1\le i\le m$ and $x\in \R^n$, it is clear that:
\begin{itemize}
\item If $0 \le t_1 \le t_2$, then $T_{t_1} \subset T_{t_2}$,
\item $T_0 = T$.
\end{itemize}

Since $T$ is a closed set, its connected components are closed. Then,
since $C$ is a compact connected component of $T$, there exists $\mu > 0$ such that $\dist(C,C') \ge 2\mu$ for every connected component $C'$ of $T$, $C'\ne C$.
Let us denote
$$
C_\mu = \{x \in \R^n \ | \ \dist(x, C) < \mu \}.
$$

\begin{lemma}\label{auxlemmacomp}
There exists $t_0 > 0$ such that for every $0 \le t \le t_0$, the connected component of $T_t$ containing $C$ is included in $C_\mu$.
\end{lemma}

\begin{proof}{Proof.}
Assume the statement does not hold. Let $(t_k)_{k \in \N}$ be a decreasing sequence of positive numbers converging to $0$ such that, if  $C'_k$ is the connected component of $T_{t_k}$ containing $C$, then $C'_k \not \subset C_\mu$.

Since $C'_k$ is connected, contains $C$ and intersects the set
$\{x \in \R^n \ | \ \dist(x, C) \ge \mu\}$, there is a point $r_k \in C'_k$ with $\dist(r_k, C) = \mu$. Since $(r_k)_{k\in \N}$ is a sequence contained in the compact set $\{x \in \R^n \ | \ \dist(x, C) = \mu\}$, it has a subsequence which converges to a point $r$ such that $\dist(r, C) = \mu$. Without loss of generality, we may assume this subsequence to be the original one.

On the other hand, since $r_{k} \in C'_{k} \subset T_{t_{k}}$,  we have that, for every $1 \le i \le m$,
$$F_i^+(t_{k}, r_{k}) \ge 0, \  \hbox{ \ and so, }  F_i^+(0, r) = \lim_{k \to \infty} F_i^+(t_{k}, r_{k}) \ge 0,$$
and, for every $1 \le i \le l$,
$$F_i^-(t_{k}, r_{k}) \le 0, \  \hbox{  \ and so, }  F_i^-(0, r) = \lim_{k \to \infty} F_i^-(t_{k}, r_{k}) \le 0.$$
This implies that $r \in T$, leading to a contradiction, since there is no point in $T$ whose distance to $C$ equals $\mu$.
\end{proof}

The following proposition shows that in order to obtain minimizers for the polynomial function $g$ on the compact connected component $C$ it is enough to consider polynomial systems with at most as many equations as variables.

\begin{proposition}\label{prop_prin} There exist $z \in C$, $S \subset \{1, \dots, m\}$ with $0 \le \# S \le n$, and $\sigma \in \{+, -\}^S$ with $\sigma_i = +$ for $l + 1 \le i \le m$,
such that $(0,z) \in \Pi_{(t,x)}(V_{S,\sigma})$ and $g(z) = \min\{g(x) \ | \ x \in C\}$.
\end{proposition}

\begin{proof}{Proof.}
Let $t_0 > 0$ be such that:
\begin{itemize}
\item for every $0 \le t \le t_0$, the connected component of $T_t$ containing $C$ is included in $C_\mu$,
\item for every $S \subset \{1, \dots, m\}, \sigma \in \{+,-\}^S$ and $t \in \Pi_t(W^{(1)}_{S,\sigma}) \cup \Pi_t(Z^{(1)}_{S,\sigma}) \cup \Pi_t(V^{(1)}_{S,\sigma})$, $t_0 < |t|$.
\end{itemize}

Let $(t_k)_{k \in \N}$ be a decreasing sequence of positive numbers converging to  $0$ with $t_1 \le t_0$.
Consider the connected component $C'_k$ of $T_{t_k}$ which contains $C$ (note that $C'_k$ is a compact set) and let $z_k$ be a point in $C'_k$ at which the function $G(t_k, \cdot )$ attains its minimum value over $C'_k$. Since the sequence $(z_k)_{k \in \N}$ is bounded, it has a convergent subsequence; without loss of generality, we may assume this subsequence to be the original one. Let $z= \lim_{k \to \infty} z_k$. Proceeding as in the proof of Lemma \ref{auxlemmacomp}, we have that $z \in C$.

In order to see that $g(z) = \min \{g(x) \ | \ x \in C\}$, note that
for every $x \in C \subset C'_{k}$, we have that $G(t_{k}, z_{k}) \le G(t_k, x)$ for every $k$; therefore,
$$
g(z) = G(0,z) =
\lim_{k \to \infty}G(t_k, z_k) \le
\lim_{k \to \infty}G(t_k, x) = G(0,x) = g(x).
$$

Now, for every $k$ and every $x \in \R^n$, at most one of the polynomials $F^+_i(t_k, x)$ and $F^-_i(t_k, x)$ may vanish, since $\tilde f_i(x) >0$. For every $k \in \N$, let $$S_k = \{i \in \{1, \dots, l\} \ | \ F^+_i(t_k, z_k) = 0 \hbox { or } F^-_i(t_k, z_k) = 0 \}\cup \{i \in \{l+1, \dots, m\} \ | \ F^+_i(t_k, z_k) = 0 \}.$$
Without loss of generality, we may assume that $S_k$ is the same set $S$ for every $k \in \N$; moreover, we may assume that, for each $i \in S$, it is always the same polynomial $F^+_i(t_k, z_k)$ or $F^-_i(t_k, z_k)$ the one which vanishes, thus defining a function $\sigma \in \{+,-\}^S$.

Since $(t_k, z_k) \in \hat W_{S,\sigma}$,  $t_k \not \in \Pi_t(W^{(0)}_{S,\sigma} \cup W^{(1)}_{S,\sigma})$ and $W_{S,\sigma} = \emptyset$ if $\#S > n$ (Lemma \ref{intersections}), we have that $\# S \le n$.
In addition, since $t_k \not \in \Pi_t(Z^{(0)}_{S,\sigma} \cup Z^{(1)}_{S,\sigma})$ and $Z_{S,\sigma} = \emptyset$ (Lemma \ref{lemma:li}), it follows that  $(t_k, z_k) \not \in \hat Z_{S,\sigma}$; therefore,
$\{\nabla_{x}F_i^{\sigma_i}(t_k,z_k), i \in S \}$ is a linearly independent set for every $k \in \N$. Finally, since the function $G(t_k, \cdot)$ attains a local minimum at the point $z_k$ when restricted to the set
$\{x \in \R^n \ | \ F_i^{\sigma_i}(t_k, x) = 0 \hbox { for every } i \in S\}$, by the Lagrange Multiplier Theorem, there exists $(\lambda_{i,k})_{i \in S}$ such that
$$\nabla_{x} G(t_k,z_k) = \sum_{i \in S} \lambda_{i,k} \nabla_{x}F_i^{\sigma_i}(t_k,z_k).$$
Therefore, $(t_k, z_k, (1,(\lambda_{i,k})_{i \in S})) \in \hat V_{S,\sigma}$;
but since $t_k \not \in \Pi_t(V^{(0)}_{S,\sigma} \cup V^{(1)}_{S,\sigma})$, we conclude that $(t_k, z_k, (1,(\lambda_{i,k})_{i \in S})) \in V_{S,\sigma}$.
Without loss of generality, we may assume that $(1,(\lambda_{i,k})_{i \in S})_{k \in \N}$ converges to a point $(\lambda_{0},(\lambda_{i,0})_{i \in S})) \in \P^{\# S}$; then $(0,z,(\lambda_{0},(\lambda_{i,0})_{i \in S})) \in V_{S,\sigma}$ and, therefore, $(0,z) \in \Pi_{(t,x)}(V_{S,\sigma})$ as we wanted to prove.
\end{proof}

\subsection{Obtaining the bounds}\label{obtthebound}

In this section we prove Lemmas \ref{intersections} and \ref{lemma:li} and we do the estimates to obtain the bounds we are looking for.

\begin{notation}
For  $p\in \Q[x_1,\dots, x_n]$ and $e \in \N$, $e \ge \deg p$,  $ \Hom (p)_e$  will denote the polynomial $ x_0^{e} p(x_1/x_0,\dots, x_n/x_0)\in \Q[x_0,\dots, x_n]$ which is obtained by homogenizing $p$ up to degree $e$.
\begin{itemize}
\item For every $1\le i \le m$,
$$\overline {F_i^{+}}(t_0,t,x_0,x)= t_0 \, \Hom(f_i)_{d}(x_0,x)+  t \, \Hom(\tilde f_i)_{d}(x_0,x) = t_0  \, \Hom(f_i)_{d}(x_0,x) +  t \left(\sum_{j=0}^n a_{ij} x_j^{d}\right),$$
$$\overline {F_i^{-}}(t_0,t,x_0,x)= t_0 \, \Hom(f_i)_{d}(x_0,x)-  t \, \Hom(\tilde f_i)_{d}(x_0,x) = t_0  \, \Hom(f_i)_{d}(x_0,x) -  t \left(\sum_{j=0}^n a_{ij} x_j^{d}\right).$$
\item For $S\subset \{1,\dots, m\}$ and $\sigma \in \{ +, -\}^S$, for every $1\le j \le n$, $$\overline G_{S,\sigma,j}(t_0,t,x_0,x, \lambda_0,\lambda)=
t_0 \ \Hom\left( \lambda_0 \frac{\partial g}{\partial x_j} - \sum_{i\in S} \lambda_i \frac{\partial f_i}{\partial x_j} \right)_{d-1} + t \ \Hom\left( \lambda_0 \frac{\partial \tilde g}{\partial x_j} - \sum_{i\in S} \lambda_i \sigma_i \frac{\partial \tilde f_i}{\partial x_j} \right)_{d-1}={}$$ $${}=
t_0 \left( \lambda_0 \Hom\Big(\frac{\partial g}{\partial x_j}\Big)_{d-1} - \sum_{i\in S} \lambda_i \Hom\Big(\frac{\partial f_i}{\partial x_j}\Big)_{d-1} \right) + t \ d x_j^{d -1} \left(\lambda_0 a_{0j} - \sum_{i\in S} \lambda_i \sigma_i a_{ij}  \right). $$
\end{itemize}
\end{notation}

\begin{proof}{Proof of Lemma \ref{intersections}.} Consider the polynomials $\overline{F_i^{\sigma_i}}$ for every $i\in S$. These polynomials are bi-homogeneous in the sets of variables $(t_0, t)$, $(x_0,x)$; therefore, they define a variety $\overline {\hat W}_{S,\sigma}$ in $\P^1 \times \P^n$ (which contains $\hat W_{S, \sigma}$ when embedded in $\P^n$). Now, the fiber $\Pi_{(t_0,t)}^{-1}(0,1)$ with respect to the projection $\Pi_{(t_0,t)}: \overline {\hat W}_{S,\sigma} \to \P^1$ is given by the set of common zeroes of the polynomials $\sum_{j=0}^n a_{ij} x_j^{d}$ for $i\in S$. But this system has no solution in $\P^n$, since, by the assumption on $A$ and the fact that $\# S >n$, the matrix $(a_{ij})_{i\in S, 0\le j \le n}$ has maximal rank $n+1$. We conclude that $\Pi_{(t_0,t)}(\overline{\hat W}_{S,\sigma})$ is not equal to $\P^1$. Since $\P^n$ is a complete variety, $\Pi_{(t_0,t)}(\overline{\hat W}_{S,\sigma})$ is closed and hence, it is a finite set. Therefore, $W_{S,\sigma} = \emptyset$.
\end{proof}

\begin{proof}{Proof of Lemma \ref{lemma:li}.} Consider the variety $\mathcal{Z}_{S,\sigma}$ defined in $\P^1\times \P^n \times \P^{\# S -1}$ by the polynomials $\overline{F_i^{\sigma_i}}$, $i\in S$, and each of the $n$ components of the vector $\sum_{i \in S} \lambda_i \nabla_{x} \overline{F_i^{\sigma_i}} $. Note that the projection to $\P^1\times \P^n$ of $\mathcal{Z}_{S,\sigma}$ contains $\hat Z_{S,\sigma}$ (when embedded in $\P^1\times \P^n$).
Consider the projection $\Pi_{(t_0,t)} : \mathcal{Z}_{S,\sigma}\to \P^1$. We will show that the fiber $\Pi_{(t_0,t)}^{-1}(0,1)$ is empty or, equivalently, that the system
$$\left\{
\begin{array}{rcl}
\displaystyle{\sum_{j=0}^n} a_{ij}x_j^{d} &=& 0\qquad i\in S \\
\displaystyle{ d x_j^{d-1} \sum_{i\in S}}  \sigma_i a_{ij} \lambda_i &= & 0\qquad j=1,\dots, n.\\
\end{array}
\right.$$
has no solution in $\P^n \times \P^{\# S -1}$. Assume, on the contrary, that $(x_0,x, \lambda)$ is a solution and let $k= \#\{ j\in \{1,\dots, n\} \mid x_j =0\}$. When specializing $x$ in the second set of equations, we get a linear equation system for $\lambda$ consisting of $n-k$ linearly independent equations in $\# S$ unknowns which has a non-trivial solution; hence $\# S\ge n+1-k$. This implies that the first $\# S$ equations do not have a common solution in $\P^n$ with $k$ vanishing coordinates.

We conclude that $\Pi_{(t_0,t)}(\mathcal{Z}_{S,\sigma})$ is not equal to $\P^1$. Since $\P^n \times \P^{\# S -1}$ is complete, as in the proof of the previous lemma, it follows that $Z_{S,\sigma} = \emptyset$.
\end{proof}

Now we use the previous constructions to derive our bounds. We will define univariate polynomials $Q_{S, \sigma}(U)$  having the minimum that $g$ takes over the compact connected components of $T$ as roots and we will obtain our bounds by means of these polynomials. Let
$$P(U,x_0,x) = Ux_0^{d_0} - \Hom(g)_{d_0}(x_0,x).$$
For $S\subset\{1,\dots, m\}$ with $\#S\le n$ and $\sigma \in \{ +,-\}^{S}$, let
$$R_{S, \sigma}(t_0,t,U) = \Res_{(x,x_0),(\lambda,\lambda_0)}(P;\overline{F_i^{\sigma_i}}, i\in S;  \overline{G}_{S,\sigma, j}, 1\le j \le n) \in \Z[t_0,t,U],$$ where $\Res_{(x,x_0),(\lambda,\lambda_0)}$ denotes the bihomogeneous resultant associated to the bi-degrees of the polynomials involved: $(d_0,0), (d, 0)$ repeated $s$ times, and $(d- 1, 1)$ repeated $n$ times.

\begin{lemma} The polynomial $R_{S,\sigma}(t_0,t,U)$ is not identically zero.
\end{lemma}
\begin{proof}{Proof.}
Let $\mathcal{S}$ be the polynomial system
$$\left\{ \begin{array}{rcl} \overline{F_i^{\sigma_i}}(t_0,t,x_0,x)=0 &  & i\in S,\\  \overline{G}_{S,\sigma, j}(t_0,t,x_0,x,\lambda_0,\lambda)=0 & & 1\le j \le n.\end{array} \right.$$

By specializing $(t_0,t)=(0,1)$ in the polynomials of the system $\mathcal{S}$, we get the following polynomial system of equations:
$$\mathcal{S}_\infty= \left\{ \begin{array}{rcl} \displaystyle\sum_{j=0}^n a_{ij} x_j^{d} = 0 & &  i \in S,\\
d x_j^{d-1} \left( a_{0j}\lambda_0-\displaystyle\sum_{i\in S}  \sigma_i a_{ij}\lambda_i \right)=0 & &  1\le j \le n. \end{array} \right. $$
We will show that $\mathcal{S}_\infty$ has finitely many solutions in $\P^{n}\times\P^{s}$, none of them lying in the hyperplane $\{x_0 =0\}$.
As a consequence of this fact, it follows that the only roots of $R_{S,\sigma}(0,1,U)$ are the finitely many values $g(x)$ where $(1,x,\lambda_0,\lambda)$ is a solution to $\mathcal{S}_\infty$; therefore $R_{S,\sigma}(t_0,t,U)$ is not identically zero.

First, note that if $(x_0,x, \lambda_0,\lambda)$ is a solution to $\mathcal{S}_\infty$, the last $n$ equations of this system imply that, for every $1\le j \le n$, either
$x_j = 0$ or $a_{0j}\lambda_0- \sum_{i\in S} \sigma_i a_{ij}\lambda_i      =0 .$
Let us show that, for every $J\subset \{1,\dots, n\}$, the system $\mathcal{S}_\infty$ has only finitely many solutions such that $x_j= 0$ if and only if $j\in J$. For a fixed $J$, these solutions are the solutions to
$$\mathcal{S}_\infty^{(1,J)}=\left\{ \sum_{j\notin J} a_{ij} x_j^{d} =0 \right. \quad  i\in S  \quad \hbox{ and } \quad \mathcal{S}_\infty^{(2,J)} = \left\{ a_{0j}\lambda_0- \sum_{i\in S} \sigma_i a_{ij} \lambda_i  = 0 \right.\quad  j\notin J.$$
Taking into account that any submatrix of $(a_{ij})$ has maximal rank, we have that:
\begin{itemize}
\item If $\#J > n-s$, the system $\mathcal{S}_\infty^{(1, J)}$
implies that $x_j = 0$ for every $j\notin J$, contradicting the definition of $J$.
\item If $\#J<n-s$, then $\mathcal{S}_\infty^{(2,J)}$ has a unique solutions $(\lambda_0,\lambda) =0$, since it consists of at least as many equations as unknowns; then, $\mathcal{S}_\infty$ has no solutions in $\P^n\times \P^s$ corresponding to $J$.
\item If $\#J = n-s$, $\mathcal{S}_\infty^{(2,J)}$ has a unique solution in $\P^s$. On the other hand, $\mathcal{S}_\infty^{(1,J)}$
has no solutions with $x_0 =0$ and exactly $d^{s}$ solutions with $x_0 =1$.
\end{itemize}
\end{proof}

Write $R_{S,\sigma}(t_0,t,U) = t^{e_{S,\sigma}}\tilde R_{S,\sigma}(t_0,t,U)$ with $e_{S,\sigma} \in \N_0$ and $\tilde R_{S,\sigma}(t_0,t,U)$ not a multiple of $t$. Note that $R_{S,\sigma}(1,t, g(x))$ vanishes on $\Pi_{(t,x)}(\hat V_{S,\sigma})$ and so, $\tilde R_{S,\sigma}(1,t,g(x))$ vanishes on  $\Pi_{(t,x)}(V_{S,\sigma})$.
Let
$$ Q_{S,\sigma}(U)= \tilde R_{S,\sigma}(1,0,U).$$

\begin{proposition}\label{boundcoeff} The polynomial $Q_{S,\sigma}(U)\in \Z[U]$ is not identically zero. The degree of $Q_{S,\sigma}(U)$ is at most $$\binom{n}{s} d^s (d -1)^{n-s},$$
where $s= \# S$, and its coefficients
 have an absolute value lower than  $$ M_{S,\sigma}= (2H_0)^{M_1} (2\widetilde H)^{sM_2 + nM_3} d^{nM_3} N_1^{M_1} \, N_2^{sM_2} \,  N_3^{n M_3} \binom{M_1 + N_1 -1}{N_1-1} \binom{M_2 + N_2 -1}{N_2-1}^s\binom{M_3 + N_3 -1}{N_3-1}^n,$$
where
\begin{itemize}
\item $\widetilde H=\max\{ H, 2n+2m\}$,
\item $M_1 = \binom{n}{s} d^s (d -1)^{n-s}$, $M_2=\binom{n}{s} d_0 d^{s-1} (d-1)^{n-s} $,
 $M_3= \binom{n-1}{s} d_0 d^{s} (d-1)^{n-s-1} $,
\item $N_1=\binom{d_0+n}{n}$,  $N_2=\binom{d +n}{n}$,   $N_3=\binom{d-1+n}{n}(s+1)$.
\end{itemize}
\end{proposition}

\begin{proof}{Proof.}
Since $\tilde R_{S,\sigma}(t_0,t,U)$ is homogeneous in the variables $t_0, t$ and it is not a multiple of $t$, it follows that $Q_{S, \sigma}(U)$ is not identically zero.

  The degree of the polynomials $f_i$ is bounded by  $d$ and their coefficients
  are of absolute value at most $H$. The corresponding quantities for $g$
  are $d_0 \leq d$ and $H_0$.
      By abuse of notation, let $A$ be an upper
  bound for the absolute values of the elements of the matrix $A$.
  From Lemma~\ref{smallmatrix}, we may assume
  $A \leq 2(n+m)$.

We deduce that  $P \in (\Z [U])[x_0, x]$ is a polynomial of degree $d_0$ and
 its coefficients are linear polynomials in $U$ with coefficients of magnitude  at most $H_0$.
Also, $\overline {F_i^{\pm}}(t_0,t,x_0,x) \in (\Z[t_0, t])[x_0, x]$  are polynomials of degree $d$
  and their coefficients are linear forms in $(t_0, t)$  with coefficients of magnitude at most
  $\widetilde H$.
  Finally, $\overline G_{S,\sigma,j}(t_0,t,x_0,x, \lambda_0,\lambda)
  \in (\Z[t_0, t])[x_0,x, \lambda_0,\lambda] $
  are bihomogeneous polynomials in $((x_0,x), (\lambda_0,\lambda))$ of degree $d-1$ in the variables $(x_0,x)$ and linear in the variables $(\lambda_0, \lambda)$, and their coefficients are
  linear forms in $(t_0, t)$ with coefficients of magnitude at most $d \widetilde H$.

  We compute the resultant $R_{S, \sigma}$ that eliminates $(x_0, x)(\lambda_0,
  \lambda)$, which is a polynomial  in $(\Z[U])[t_0, t]$. Recall that
  the bihomogeneous resultant $\text{Res}_{(x_0,x), (\lambda, \lambda_0)}$ of a bihomogeneous system of $n+s+1$ polynomials
  consisting of a polynomial of bidegree $(d_0,0)$, $s$ polynomials of bidegree $(d, 0)$ and $n$ polynomials of bidegree
  $(d - 1, 1)$ is a multihomogeneous polynomial of degree
   $$M_1= \text{Bez}( (d, 0), s; (d - 1, 1), n)=\binom{n}{s} d^s (d -1)^{n-s} $$
    in the coefficients
  of the polynomial of bidegree $(d_0,0)$, of degree
  $$M_2= \text{Bez}( (d_0,0), 1; (d, 0), s-1; (d - 1, 1), n) = \binom{n}{s} d_0 d^{s-1} (d-1)^{n-s} $$
   in the coefficients of each of the $s$ polynomials
  of bidegree $(d, 0)$, and of degree
  $$M_3=\text{Bez}( (d_0,0),1;(d, 0), s; (d - 1, 1), n-1) = \binom{n-1}{s} d_0 d^{s} (d-1)^{n-s-1} $$ in the coefficients of each of  the $n$ polynomials of
  bidegree $(d-1,1)$. Here $\text{Bez}(\bm{d}_1,s_1; \dots; \bm{d}_r, s_r)$ denotes the B\'ezout number of a bihomogeneous system formed by $s_i$ polynomials of bi-degree $\bm{d}_i=(d_{i,1}, d_{i,2})$ for $1\le i \le r$ (see \cite[Chapter IV, Sec. 2]{Shaf}).

It follows that $R_{S, \sigma}$ is a sum of terms of the form
  \begin{equation}\label{eq:resterms}
    \rho \, \alpha \, \prod_{i\in S} \beta_i \, \prod_{1\le j \le n} \gamma_j,
  \end{equation}
  where $\rho \in \Z$  is a coefficient of the bihomogeneous resultant $\text{Res}_{(x_0,x), (\lambda, \lambda_0)}$,
$\alpha$ denotes a monomial in the
  coefficients of $P$ of total degree $M_1$,  $\beta_i$
  denotes a monomial in the coefficients of $\overline {F_i^{\pm}}$ of
  total degree $M_2$ for every $i\in S$, and $\gamma_j$
  denotes a monomial in the coefficients of $\overline G_{S,\sigma,j}$
  of total degree $M_3$ for every  $1\le j \le n$.
In particular, the degree of $R_{S, \sigma}$ in the variable $U$ is at most $M_1$. % and, therefore, the degree of $Q_{S, \sigma}(U)$ is at most $M_1$.

Note that the polynomial $Q_{S, \sigma}(U)$ is the coefficient of $R_{S, \sigma}\in (\Z[U])[t_0, t]$ corresponding to the smallest power of $t$. Therefore,   $$\deg Q_{S,\sigma}(U)\le \deg_U R_{S, \sigma}(U,t_0,t) \le M_1 = \binom{n}{s} d^s (d -1)^{n-s}.$$

In order to estimate the magnitude of its coefficients, we may set $t_0=1$ in  $R_{S, \sigma}$  and, by abuse of notation, write $R_{S, \sigma}$ for the specialized polynomial.
For every $k$, we will compute an upper bound for the magnitude of the  coefficients of the polynomial in $\Z[U]$ that appears as coefficient of
$t^k$ in $R_{S,\sigma}$.

First, we apply \cite[Theorem 1.1]{Sombra04} to bound the coefficients $\rho \in \Z$ of the resultant $\text{Res}_{(x_0,x),(\lambda_0,\lambda)}$.  We obtain:
  \begin{equation}\label{eq:boundcoeff}
    \abs{\rho} \leq N_1^{M_1} \, (N_2^{M_2})^{s} \, ( N_3^{M_3})^{n},
  \end{equation}
  where $N_1=\binom{d_0+n}{n}$ and $N_2=\binom{d +n}{n}$ are the cardinalities of the
   supports of generic homogeneous polynomials of degrees $d_0$ and $d$ respectively, and $N_3=\binom{d-1+n}{n}(s+1)$ is the cardinality of
   the support of a generic bihomogeneous polynomial of bidegree $(d - 1, 1)$ in $(x_0,x), (\lambda_0, \lambda)$.

Note that $\alpha\in \Z [U] $ is a polynomial in $U$  with integer coefficients and degree bounded by $M_1$, and the absolute value of the coefficient of the power $U^j$ in $\alpha^{M_1}$ is at most
\begin{equation}\label{eq:boundU}
\binom{M_1}{j} H_0^{M_1 - j} <  (2H_0)^{M_1}.
\end{equation}

On the other hand, the product $\prod_{i\in S} \beta_i \, \prod_{1\le j \le n} \gamma_j\in \Z[t]$ is a polynomial in the variable $t$ with integer coefficients, which is a product of $sM_2$ linear factors that are coefficients of the $\overline{F_i^{\pm}}$ and $nM_3$ linear factors that are coefficients of the $\overline G_{S,\sigma,j}$. Thus, for a fixed $k$, using the upper bounds on the coefficients of the polynomials $\overline{F_i^{\pm}}$ and $\overline G_{S,\sigma,j}$, it follows that the coefficient of $t^k$ in this product is at most
\begin{equation}\label{eq:boundt}
\binom{sM_2 + nM_3}{k} \widetilde H^{sM_2} (d \widetilde H)^{nM_3}< (2\widetilde H)^{sM_2 + nM_3} d^{nM_3}.
\end{equation}

Finally, taking into account the multihomogeneous structure of the resultant, it follows that $R_{S,\sigma}$ is a sum of at most
\begin{equation}\label{eq:boundterms}
 \binom{M_1 + N_1 -1}{N_1 - 1} \binom{M_2 + N_2 -1}{N_2 - 1}^s\binom{M_3 + N_3 -1}{N_3-1}^n
\end{equation}
terms of the form (\ref{eq:resterms}).

Combining the upper bounds (\ref{eq:boundcoeff}), (\ref{eq:boundU}), (\ref{eq:boundt}) and (\ref{eq:boundterms}) we obtain the stated upper bound for the absolute value of the coefficients of $Q_{S, \sigma}(U)$.
\end{proof}

We can prove now the main result of the paper:

\begin{proof}{Proof of Theorem \ref{minimum}.}
By Proposition \ref{prop_prin}, the polynomial $g$ attains its minimum value over $C$ at a point $z_0\in C$ such that $(0,z_0) \in \Pi_{(t,x)}(V_{S,\sigma})$ for certain
$S \subset \{1, \dots, m\}$ with $0 \le \# S \le n$, and $\sigma \in \{+, -\}^S$ with $\sigma_i = +$ for $l + 1 \le i \le m$. Now, for every $(0,z) \in \Pi_{(t,x)}(V_{S,\sigma})$, we have that $Q_{S,\sigma}(g(z))= 0$.

Then, if $s=\# S$, Proposition \ref{boundcoeff} implies that $g(z_0)$ is an algebraic number of degree at most $ \binom{n}{s} d^s (d -1)^{n-s}\le 2^{n-1} d^n$. Furthermore, if $g(z_0) \ne 0$,  its absolute value is greater than or equal to $M_{S, \sigma}^{-1}$ (see \cite[Proposition 2.5.9]{MignotteStefanescu}).

We keep the notation in Proposition \ref{boundcoeff}. In order to get the stated bound for the minimum, we use the following facts:

\begin{itemize}

\item $N_1, N_2 \le \frac32 d^n$,

\item $N_3 \le \frac94 d^n$:
for $n = 2$ and $n=3$ the bound holds easily, for $n \ge 4$,
$$
N_3 \le (n+1)\prod_{i=1}^n \frac{ d-1+i}{i}
\le (n+1) d \Big(\frac{\sum_{i=2}^n \frac{ d-1+i}{i}}{n-1}\Big)^{n-1}
\le (n+1) d \Big(( d-1)\frac{\log(n)}{n-1} + 1\Big)^{n-1} \le
$$
$$
\le (n+1) d \Big(( d-1)0.47 + 1\Big)^{n-1}
\le (n+1)0.74^{n-1} d^{n} \le \frac94  d^{n}.
$$

\item $\binom{M_i + N_i -1}{N_i - 1} \le 2^{M_i + N_i}$ for $1\le i \le 3$.

\end{itemize}

Then we have
$$
M_{S, \sigma} \le
2^{ 2(M_1 + sM_2 + nM_3) + (\log_2(3) - 1)(M_1 + sM_2) + 2(\log_2(3) - 1) n M_3 + N_1 + sN_2 + nN_3 } \ \cdot $$ $$\cdot \
\tilde H^{ M_1 + sM_2 + nM_3}
 d^{n(M_3 +  M_1 + sM_2 + nM_3  )  }.
$$

Since
$M_1 + sM_2 + nM_3 \le (n+1)\binom{n}{s}d^{n} \le (n+1)2^{n-1}d^{n}$ and $M_3 \le 2^{n-2}d^n$, we have
$$
M_{S, \sigma} \le
2^{ ( ( (3\log_2(3) + 1)n + 2\log_2(3) + 2)2^{n-2} + \frac32(n+1) + \frac94n )d^n}
\tilde H^{ (n+1)2^{n-1}d^n}
d^{ (2n^2 +  3n) 2^{n-2}d^n},
$$
and taking into account that $\tilde H \ge 6$ and $d \ge 2$, we obtain
$$
M_{S, \sigma} \le
2^{ (
( -2n^2  + (\log_2(3) + 2)n + 4\log_2(3) + 4)2^{n-2}
 + \frac32(n+1) + \frac94n )d^n}
\tilde H^{ n2^{n}d^n}
d^{ n^22^{n}d^n}.
$$

Finally, the result holds since for $n \ge 2$,
$$ ( -2n^2 +  (\log_2(3) + 2)n +   4\log_2(3) + 4 )2^{n-2} + \frac32(n+1) + \frac94n  \le \Big(4-\frac{n}{2}\Big)n2^n.$$
\end{proof}

\begin{remark}
The algebraic degrees of the coordinates of a minimizer are also bounded by  $2^{n-1} d^n $. This can be seen simply by replacing the polynomial $g$ by a coordinate $x_i$ in the previous construction, namely, taking $P(U,x_0,x) = Ux_0 - x_i$.
\end{remark}

In applications (see Section \ref{sec:distance}), sometimes the minimization of a polynomial $g$ needs to be done over a component not necessarily compact, but with a compact set of minimizers. The result in Theorem \ref{minimum} can be extended to this situation:

\begin{theorem}\label{minimum_noncompact}
Let $T = \{x \in \R^n \ | \ f_1(x) = \dots = f_{l}(x) = 0, f_{l + 1}(x) \ge  0, \dots, f_{m}(x) \ge  0\}$ be defined by polynomials $f_1,\dots, f_m \in \Z[x_1,\dots, x_n]$ with degrees bounded by $d$ and coefficients of absolute value at most $H$, and let $C$ be a connected component of $T$. Let $g\in \Z[x_1,\dots, x_n]$ be a polynomial of degree $d_0\le d$ and coefficients of absolute value bounded by $H_0\le H$, and let $g_{\min, C}$ the minimum value that $g$ takes over $C$. Assume that the set
$$C_{\min}= \{ z \in C \mid g(z) = g_{\min, C}\}$$
is compact.
Then, $g_{\min, C}$ is an algebraic number of degree at most
$2^{n-1} d^n $
and,
if it is not zero, its absolute value is greater or equal to $(2^{4-\frac{n}2}\tilde H d^n)^{-n2^nd^n}$.
\end{theorem}

\begin{proof}{Proof.}
Take $M\in \R$  so that $C_{\min}\subset B(0,M)$, and let $C'$ be the connected component of the set
$$
T' = \{
x \in \R^{n} \mid
f_1(x) = \dots = f_{l}(x) = 0, f_{l + 1}(x) \ge  0, \dots, f_{m}(x) \ge  0,
(M + 1)^2 - \sum_{i=1}^nx_i^2\ge 0
\}
$$
which contains $C_{\min}$. Note that $C'$ is compact, since $T'$ is bounded.

By Proposition \ref{prop_prin}, there exist $z_0 \in C'$, $S \subset \{1,\dots, m + 1\}$ with $\#S \le n$ and $\sigma \in \{+,-\}^S$ such that $(0,z_0)
\in \Pi_{(t,x)}(V_{S,\sigma})$ (for the corresponding variety $V_{S,\sigma}$ associated to the equations of $T'$) such that $g(z_0)= g_{\min,C}$.

Since $(M + 1)^2 - \sum_{i=1}^n (z_{0,i})^2 \ne 0$, we have that
$S \subset \{1,\dots, m\}$ (see the proof of Proposition \ref{prop_prin}). Now the result follows by Proposition \ref{boundcoeff}, proceeding as in the proof of Theorem \ref{minimum}.
\end{proof}

\section{Bounds for the separation between disjoint connected components of basic closed semialgebraic sets}\label{sec:distance}

In this section we will apply our previous results to the case when $g$ is the square of the Euclidean distance in order to
obtain bounds for the separation between two disjoint (and at least one compact) connected components of semialgebraic sets defined by non-strict inequalities. In particular, this gives a separation bound for two connected components of a closed semialgebraic set provided that one of them is compact.

\begin{proof}{Proof of Theorem \ref{distance}.}
 We have that  $C_1 \times C_2$ is a connected component of the set $T_1 \times T_2 = \{ (x,y) \in \R^{2n} \ | \
f_1(x) = \dots = f_{l_1}(x) = 0, f_{l_1 + 1}(x) \ge  0, \dots, f_{m_1}(x) \ge  0,
g_1(y) = \dots = g_{l_2}(y) = 0, g_{l_2 + 1}(y) \ge  0, \dots, g_{m_2}(y) \ge  0
\}$, and if $D(x,y) = \sum_{i = 1}^n(x_i - y_i)^2$, then the minimum value that $D$ takes over $C_1\times C_2$ equals $\dist^2(C_1, C_2)>0$.
In addition, the set$ \{(x,y) \in C_1 \times C_2 \ | \ \dist(x,y) = \dist(C_1, C_2)\}$ is bounded and, therefore, compact.
Then, the result follows from Theorem \ref{minimum_noncompact}.
\end{proof}

\begin{example}
Consider $d, H, n \in \N$ with even $d$ and $f_1, \dots, f_n \in \Z[x_1, \dots, x_n]$ defined by
$$
f_1(x) = Hx_1 - 1,  \quad \quad  f_i(x) = x_{i} - x_{i-1}^d \ \hbox{ for } 2\le i \le n-1, \quad \quad f_n(x) = x_n^2 - x_{n-1}^{d}.
$$
The set $\{x \in \R^n \ | \ f_1(x) = \dots = f_n(x) = 0\}$ equals $\{p,q\}$ with
$$
p= (H^{-1}, H^{-d}, \dots, H^{-d^{n-2}}, H^{-\frac{1}2d^{n-1}}), \quad \quad
q = (H^{-1}, H^{-d}, \dots, H^{-d^{n-2}}, -H^{-\frac{1}2d^{n-1}})
$$
and the distance between $p$ and $q$ is $2H^{-\frac{1}2d^{n-1}}$. This shows that the double exponential nature of our bound is unavoidable even in the case of different connected components of a single closed semialgebraic set.
\end{example}

\end{document}